\documentclass[twoside,11pt,reqno]{amsart}
\usepackage{amsmath,amssymb,amscd,mathrsfs,amscd, todonotes}
\usepackage{graphics,verbatim}
\usepackage{todonotes}
\usepackage{hyperref}

\newtheorem{thm}{Theorem}[section]
\newtheorem{cor}[thm]{Corollary}
\newtheorem{lem}[thm]{Lemma}

\theoremstyle{definition}

\newtheorem{rem}[thm]{Remark}

\newcommand{\Ext}{\operatorname{Ext}}
\newcommand{\Tor}{\operatorname{Tor}}

\newcommand{\Hom}{\operatorname{Hom}}

\newcommand{\HH}{\operatorname{H}}

\newcommand{\Lie}{\mathsf{Lie}}

\numberwithin{equation}{section}

\begin{document}

\title[The Lie module and its complexity]
{\bf The Lie module and its complexity}

\author{\sc Frederick R. Cohen}

\address
{Department of Mathematics\\University of Rochester\\Hylan
Building\\Rochester, NY~14627, USA}
 \email{cohf@math.rochester.edu}

\author{\sc David J. Hemmer}
\address
{Department of Mathematics\\ University at Buffalo, SUNY \\
244 Mathematics Building\\Buffalo, NY~14260, USA}
\thanks{Research of the second author was supported in part by NSF
grant  DMS-1068783} \email{dhemmer@math.buffalo.edu}

\author{\sc Daniel K. Nakano}
\address
{Department of Mathematics\\ University of Georgia \\
Athens\\ GA~30602, USA}
\thanks{Research of the third author was supported in part by NSF
grant  DMS-1402271} \email{nakano@math.uga.edu}

\date{February 2015}

\subjclass[2000]{Primary 20C30, 55S12, 55P47}

\begin{abstract}

The complexity of a module  is an important homological invariant that measures the polynomial rate of growth of its minimal projective resolution.
For the symmetric group $\Sigma_n$, the Lie module $\Lie (n)$ has attracted a great deal of interest in recent years.
We prove here that the complexity of $\Lie(n)$ in characteristic $p$ is $t$ where $p^t$ is the largest power of $p$ dividing $n$, thus proving a conjecture of Erdmann, Lim and Tan.
The proof uses work of Arone and Kankaanrinta which describes the homology $\HH_\bullet(\Sigma_n, \Lie(n))$ and earlier work of Hemmer and Nakano on complexity for
modules over $\Sigma_n$ that involves restriction to Young subgroups.
\end{abstract}

\maketitle

\section{Introduction}

Let $G$ be a finite group and $k$ be an algebraically closed of characteristic $p$. In 1977, Alperin \cite{AlperinCompl} defined the \emph{complexity} of $M$, denoted $c_{G}(M)$, as the rate of growth of the minimal projective resolution of $M$. It is in some sense a measure of how far $M$ is from being projective; in particular $M$ is projective if and only if $c_{G}(M)=0$. Alperin's definition naturally led to the the theory of
support variety of modules, in addition to, the emphasis on homological and topological methods in representation theory.  With such methods one can compute the
complexity of modules without explicitly describing the minimal resolution. For an introduction to complexity and support varieties we refer to \cite{BensonbookII}.

Let $\Sigma_{n}$ be the symmetric group on $n$-letters. The representation theory and combinatorics of the symmetric group via Young tableaux has been well-studied.
In \cite{HN}, Hemmer and Nakano provided a combinatorial formula for the complexity of Young modules over $\Sigma_{n}$ via the
removal of horizontal $p$-hooks. They also determined the complexity for completely splittable irreducible modules. One of their main tools was the development of
a formula for $c_{\Sigma_{n}}(M)$ via branching to Young subgroups.

The goal of this paper is to compute the complexity of the $k\Sigma_n$ module $\Lie(n)$, which we define next. For any commutative ring $R$ and positive integer $n$, let $\Lie_R(x_1, x_2, \ldots, x_n)$ be the free Lie algebra over $R$ generated by $x_1, x_2, \ldots, x_n$ and let $\Lie_R(n)$ be the submodule spanned by all bracket monomials containing each $x_i$ exactly once.  Then $\Lie_R(n)$ is a module for the symmetric group $\Sigma_n$ acting by permuting the variables.

We will be interested in $\Lie_k(n):=\Lie(n)$. This module arises naturally in topology, for example as the top degree homology of the configuration space of $n$ points in the plane
tensored by the sign representation. In characteristic zero there is a beautiful  description of its complex character in terms of tableaux combinatorics,
see \cite[Chapter 8]{Reutenauerbook} for a thorough treatment. Furthermore,  the representation $\Lie(n)$ is a direct summand of ${\mathbb Q}\Sigma_n$.
In characteristic $p$, very little is known about the module structure of $\Lie(n)$ except in special cases, for example small $n$ or when $p^2 \nmid n$.
Over an arbitrary field $k$, $\Lie(n)$ has dimension $(n-1)!$ and is free over $k\Sigma_{n-1}$.

Erdmann, Lim and Tam \cite{ErdmannLimTan} stated a conjecture for $c_{\Sigma_{n}}(\Lie(n))$. Our strategy in proving this conjecture involves first employing the Hemmer-Nakano
formula for the complexity of $k\Sigma_{n}$-modules via branching to Young subgroups. In the next section, we provide a homology version of this result. This lends itself well to applying the
calculations due to Arone and Kankaanrinta \cite{aroneKank} to give estimates on the rate of growth of the homology groups in the aforementioned complexity formula.

The authors acknowledge the hospitality of the Mathematics Institute of Extended Stay America where a majority of these results were obtained.

\section{Complexity: Interpretations via cohomology and homology}\label{sec:complexityinterp}

Throughout the paper $G$ will denote a finite group, $k$ an algebraically closed field, and $kG$ the group algebra for $G$. All $kG$-modules considered
are finite-dimensional. If $M$ is a $kG$-module then $M^{*}$ will denote the dual $kG$-module.

Let $V_\bullet= \{V_t : t \in \mathbb{N}\}$ be a sequence of finite-dimensional vector spaces. The {\em rate of growth} of $V_\bullet$, denoted $\gamma(V_\bullet)$, is the smallest positive integer $c$ such that $\operatorname{dim} V_t \leq K \cdot t^{c-1}$ for some constant $K$ and for all $t$. For example $\gamma(V_\bullet)=1$ if and only if the dimensions of $V_t$ are uniformly bounded.

We define suspension $\Sigma^{1}V_{\bullet}$ to be $(\Sigma^{1}V_{\bullet})_{t+1}=V_{t}$. In particular, the degree of $\Sigma^{1}V_{t}$ is $t+1$. This corresponds algebraically to the behavior of the topological suspension of a space on the level of homology. One can now easily show that
$\gamma(V_{\bullet})=\gamma(\Sigma^{1}V_{\bullet})$.  Let $c=\gamma(V_{\bullet})$ and $e=\gamma(\Sigma^{1}V_{\bullet})$. Observe that $\dim V_{t} \leq K\cdot t^{c-1}$ so
$\dim (\Sigma^{1}V_{\bullet})_{t}=\dim V_{t-1} \leq  K\cdot (t-1)^{c-1}$, thus $e\leq c$.

We have $\dim (\Sigma^{1}V_{\bullet})_{t} \leq Q\cdot t^{e-1} $ for some positive constant  $Q$. Then
$$\dim V_{t} \leq Q\cdot (t+1)^{e-1} =Q\cdot  t^{e-1} + Q\cdot p(t).$$
where $p(t)$ is a polynomial of degree $e-2$ with strictly positive coefficients. Therefore, there
exists $Q^{\prime}>0$ such that $p(t)\leq Q^{\prime}\cdot t^{e-1}$, and $\dim V_{t}\leq (Q+Q^{\prime})\cdot t^{e-1}$. Consequently,
$c\leq e$, and $c=e$. We have proven the following:

\begin{lem}
\label{lem: suspensionpreservesrog} Let $V_\bullet= \{V_t : t \in \mathbb{N}\}$ be a sequence of finite dimensional vector spaces. Then for any $i \geq 0$ we have:

$$\gamma(V_\bullet)=\gamma(\Sigma^i V_\bullet).$$
\end{lem}

Let $M$ be a $kG$-module and let $P_{\bullet}\rightarrow M$ be a minimal projective resolution of $M$. The complexity of $M$, denoted by $c_{G}(M)$, is $\gamma(P_\bullet)$. Since
$kG$ is a self-injective algebra, $c_{G}(M)=0$ if and only if $M$ is a projective $kG$-module.

We begin by following \cite[Section 2.4]{EvensBook}. Suppose $P_\bullet \rightarrow M$ is a minimal projective resolution and $S$ is a simple $kG$ module.
Then the differentials in the complexes defining $\Ext$ and $\Tor$ vanish and one gets:

\begin{equation}
\label{eq: tor special case}
\Tor_n^{kG}(M,S) = P_n \otimes_{kG} S
\end{equation}
and
\begin{equation}
\label{eq: ext special case}
\Ext^n_{kG}(M,S)=\Hom_{kG}(P_n,S).
\end{equation}
Note that in \eqref{eq: tor special case} one considers $P_n$ as a right $kG$-module via the usual $p \cdot g :=g^{-1}p.$ The dimension of $\Hom_{kG}(P_n,S)$ is the number of summands of $P_n$ isomorphic to the projective cover of $S$. Using the usual adjoint associativity between $\Hom$ and tensor product we obtain:

\begin{eqnarray*}
  \dim_k(P_n \otimes_{kG} S) &=& \dim_k\Hom_k(P_n \otimes_{kG} S,k) \\
   &=& \dim_k \Hom_{kG}(P_n, \Hom_k(S,k))\\
   &=& \dim_k \Hom_{kG}(P_n, S^{*}).
\end{eqnarray*}
Therefore, we obtain the following relationship:
\begin{equation}\label{Ext-Tor-relationship}
\dim_k\Ext^n_{kG}(M,S) = \dim_k\Tor^{kG}_n(M,S^*),
\end{equation}
where $S$ is a simple $kG$-module and $M$ is an arbitrary $kG$-module.

We will now consider the case of modules over the symmetric group on $n$-letters, $\Sigma_{n}$.
Let $\lambda=(\lambda_{1},\lambda_{2},\dots,\lambda_{l})\vDash n$ be a composition of $n$. The associated
Young subgroup is $\Sigma_{\lambda}\cong \Sigma_{\lambda_{1}}\times \Sigma_{\lambda_{2}}\times \dots \times
\Sigma_{\lambda_{l}}$.  We apply the results of \cite{HN}, which let one interpret the complexity of a $k\Sigma_n$ module
in terms of its homology or cohomology on restriction to Young subgroups.

\begin{thm}
\label{symmetric-complexity-branch}
Let $M$ be a $k\Sigma_{n}$-module. The following are equivalent.
\begin{itemize}
\item[(a)] $c_{G}(M)$
\item[(b)] $\operatorname{max}_{\lambda\vDash n} \{\gamma(\HH^\bullet(\Sigma_{\lambda},M))\}$
\item[(c)] $\operatorname{max}_{\lambda\vDash n} \{\gamma(\HH_{\bullet}(\Sigma_{\lambda},M))\}$
\end{itemize}
\end{thm}

\begin{proof} The statement that (a) is equivalent to (b) is \cite[Theorem 4.3.1]{HN}. To
prove that (b) is equivalent to (c), first observe that for $n\geq 0$,
\begin{eqnarray*}
\dim \HH^{n}(\Sigma_{\lambda},M) &=& \dim \Ext^{n}_{k\Sigma_{\lambda}}(k,M) \\
                                                        &=& \dim \text{Ext}^{n}_{k\Sigma_{\lambda}}(M^{*},k) \\
                                                        &=& \dim \Tor_{n}^{k\Sigma_{\lambda}}(M^{*},k) \\
                                                        &=& \dim \HH_{n}(\Sigma_{\lambda},M^{*}).
\end{eqnarray*}
In order to complete the proof, use the fact that $c_{\Sigma_{\lambda}}(M)=c_{\Sigma_{\lambda}}(M^{*})$ (cf. \cite[Prop. 5.7.3]{BensonbookII}).
\end{proof}

\section{The computation of $c_{\Sigma_{n}}(\Lie(n))$}

We first reduce the computation of the complexity of the Lie module $\Lie(n)$ to the computation
of $\gamma(\HH_{\bullet}(\Sigma_{p^{r}},\Lie(p^{r})))$ for all $r\leq t$, where $p^{t}$ is the largest power of $p$ dividing $n$.
This will be accomplished by using the work in \cite{aroneKank} in conjunction with Theorem~\ref{symmetric-complexity-branch}(c).

Let $\lambda=(\lambda_{1},\lambda_{2},\dots,\lambda_{l})\vDash n$ and set $\hat{\lambda}=\text{gcd}(\lambda_{1},\lambda_{2},\dots,\lambda_{l})$.
Moreover, let $j(\lambda)$ be the largest integer such that $p^{j(\lambda)}$ divides $\hat{\lambda}$. According to \cite[p. 4]{aroneKank},
\begin{equation}
\HH_{\bullet}(\Sigma_{\lambda_{1}}\times \Sigma_{\lambda_{2}}\times \dots \times \Sigma_{\lambda_{l}},\Lie(n))
\cong \bigoplus_{r=0}^{j(\lambda)} \HH_{\bullet}(\Sigma_{p^{r}},\Lie(p^{r}))^{\oplus C_{r}}
\end{equation}
where $C_{r} \geq 1$ only depends on $p$, $\lambda$ and $n$. In particular the value of the various $C_r$ does not change the rate of growth. We can conclude that
\begin{equation}\label{rate-reduction}
\gamma(\HH_{\bullet}(\Sigma_{\lambda_{1}}\times \Sigma_{\lambda_{2}}\times \dots \times \Sigma_{\lambda_{l}},\Lie(n)))=\text{max}_{0\leq r \leq j(\lambda)}
\gamma(\HH_{\bullet}(\Sigma_{p^{r}},\Lie(p^{r}))).
\end{equation}
Finally, one can apply Theorem~\ref{symmetric-complexity-branch} with (\ref{rate-reduction}) to deduce the following result.

\begin{thm}\label{thm:reduction}  Let $t$ be the largest positive integer such that $p^{t}\mid n$. Then
$$c_{\Sigma_{n}}(\Lie(n))=\operatorname{max}_{0\leq r \leq t} \gamma(\HH_{\bullet}(\Sigma_{p^{r}},\Lie(p^{r}))).$$
\end{thm}

Let
$$M_r=\HH_\bullet(\Sigma_{p^r},\Lie(p^r))= \Tor_\bullet^{k\Sigma_{p^r}}(k,\Lie(p^r)).$$
Arone and Kankaanrinta gives a basis for the $(r+1)$st suspension $\Sigma^{1+r}M_r$
in terms of ``completely inadmissible Dyer-Lashof words of length $r$". Their results are summarized in the next theorem.

\begin{thm}\cite[Thm. 3.2]{aroneKank}
\label{thm:AKmaintheorem}
The following elements constitute a basis for $\Sigma^{1+r}M_r$:

if $p>2$

$$\{\beta^{\epsilon_1}{Q^{s_1}}\cdots \beta^{\epsilon_r}{Q^{s_r}}u \mid s_r \geq 1, s_j >ps_{j+1}-\epsilon_{j+1} \forall 1 \leq j <r\},$$

if p=2

$$\{Q^{s_1} \cdots Q^{s_r}u \mid s_r \geq 1, s_j >2s_{j+1} \forall 1 \leq j <r\}.$$

Here $u$ is of dimension 1, the $Q^{s_j}$s are Dyer-Lashof operations and the $\beta$s are the homology Bocksteins. Thus $Q^s$ increases dimension by $s$ if $p=2$ and by $2s(p-1)$ if $p>2$, and $\beta$ decreases dimension by one.
\end{thm}
As a special case we remark that the basis element $Q^{s_1}Q^{s_2} \cdots Q^{s_r}u$
lies in degree $2(p-1)(s_1+s_2+ \cdots + s_r)$ for p odd and $(s_1+s_2+ \cdots + s_r)$ for $p=2$. 

\begin{rem} The operations in Theorem~\ref{thm:AKmaintheorem} are not required to satisfy the Adem relations.
For example, $Q^4Q^1(u)$ vanishes by the Adem relations but is non-zero in the 
current setting of the homology.
\end{rem}

The results above are now employed to compute $\gamma(\HH_{\bullet}(\Sigma_{p^r},\Lie(p^{r})))$. 

\begin{thm}\label{thm:Liepr-growth}  For all $r\geq0$,  $\gamma(\HH_{\bullet}(\Sigma_{p^r},\Lie(p^{r})))=r$.
\end{thm}

\begin{proof} As observed in Lemma~\ref{lem: suspensionpreservesrog}, the rate of growth does not change by taking suspensions so it
suffices to look at the rate of growth of $\Sigma^{1+r}M_r$ where a basis is given in Theorem~\ref{thm:AKmaintheorem}.

We first prove that $\gamma(\HH_{\bullet}(\Sigma_{p^r},\Lie(p^{r})))\leq r$. We do this by embedding the basis of $\HH_{m}(\Sigma_{p^r},\Lie(p^{r}))$, for each $m$, in a sequence of larger vector spaces which have rate of growth $r$.

For $p=2$, using Theorem~\ref{thm:AKmaintheorem}, we
see that the number of basis elements is at most the number of monomials of degree $m$ in the various $Q$s. This is bounded by the the number of compositions of $m$ of the form
$\mu=(s_{1},s_{2},\dots,s_{r})\vDash m$. The number of such compositions has rate of growth $r$, because if coincides with the Krull dimension of the polynomial ring in
$r$ variables. 

For $p$-odd the number of basis elements in degree $m$ is bounded by the number of monomials of degree $m$ in the ring:
\begin{equation}
\label{eq: polytensorext}
\Lambda(e_1, \ldots, e_r) \otimes k[x_1, x_2, \ldots, x_r].
\end{equation}
The exterior algebra generators in \eqref{eq: polytensorext} have degree -1 and correspond to the Bocksteins while the polynomial generators have degree $2(p-1)$ and correspond to the $Q$'s. It is clear that taking the tensor product with the finite dimensional algebra $\Lambda(e_1, \ldots, e_r)$ does not change the rate of growth, nor does shifting the generators $x_i$ from degree 1 to degree $2(p-1)$. So the rate of growth is again bounded above by $r$, the Krull dimension of the polynomial ring.
Therefore, in the $p$-odd case, one also has $\gamma(\HH_{\bullet}(\Sigma_{p^r},\Lie(p^{r})))\leq r$.

Next we show that $\gamma(\HH_{\bullet}(\Sigma_{p^r},\Lie(p^{r})))\geq r$ by constructing a sufficiently large number of basis vectors in each dimension. Fix $p$ and $r$ and let $x \in {\mathbb N}$. Define:
    \begin{equation}
        \label{eq: def of i}
            i=2(p-1)(p^{r-1}+2p^{r-2} +3p^{r-3} + \cdots + (r-1)p +r-1)x.
    \end{equation}

One can describe a subset of the basis for $\HH_i(\Sigma_{p^r}, \Lie(p^r))$ which has $x^{r-1}$ elements. Since $i=Cx$ where $C$ is a constant, depending only on $p$ and $r$, this will prove that $\gamma(\HH_\bullet(\Sigma_{p^r}, \Lie(p^r)) \geq r$ as desired.

We proceed by constructing $x^{r-1}$ distinct basis elements which all contribute to degree $i$. We do not use any of the basis elements containing Bocksteins, so the construction is the same for $p=2$ and $p$ odd. In either case the $s_j$'s must satisfy $s_r \geq 1$ and $s_j >ps_{j+1}$ for $1 \leq j <r$. We choose $s_1, s_2, \ldots, s_r$ as follows. First choose $s_r$ so that:
$$1 \leq s_r \leq x.$$ Now we need $s_{r-1}>ps_r$ so choose $s_{r-1}$ such that
$$px+1 \leq s_{r-1} \leq px+x.$$ Proceed in this way for $s_{r-2}, s_{r-3}, \ldots, s_2$:

$$\begin{array}{rcl}
 p^2x+px+1 \leq& s_{r-2}&\leq  p^2x+px+x\\
 p^3x+p^2x+px+1 \leq& s_{r-3}& \leq p^3x+p^2x+px+x\\
 &\vdots&\\
 p^{r-2}x+p^{r-3}x + \cdots + px +1 \leq & s_2\,\,\, &\leq  p^{r-2}x+p^{r-3}x + \cdots + px +x
\end{array}$$

Observe that there were $x$ choices for each of $s_r, s_{r-1}, \ldots, s_2$. Finally choose:

\begin{multline}
\label{eq: defining s1}
s_1:=(p^{r-1}x+p^{r-2}x + \cdots +px+x) + (x-s_r)+ (px+x-s_{r-1})\\ + (p^2x+px+x -s_{r-2}) +\cdots + (p^{r-2}x+p^{r-3}x + \cdots + px+x -s_2).
\end{multline}

Observe that $s_1 >ps_2$ so $\{s_1, s_2, \ldots, s_r\}$ is an allowable sequence. From \eqref{eq: defining s1} it is clear that:

\begin{equation}
\label{eq: sum of si}
s_1+s_2+ \cdots + s_r=(p^{r-1}+2p^{r-2} + 3p^{r-3} + \cdots +(r-1)p+(r-1))x.
\end{equation}

Since there were $x$ choices for each of $s_2, s_3, s_4, \ldots, s_r$ and then $s_1$ is determined, so we have $x^{r-1}$ distinct basis elements. From \eqref{eq: sum of si} and
the remark after Theorem~\ref{thm:AKmaintheorem}, for $p$ odd these basis elements all occur in degree $i$ and for $p=2$ they occur in degree $i/2$. In either case $i$ is just a linear function of $x$ and the dimension of $\HH_i$ is at least $x^{r-1}$, giving the result.
\end{proof}

By combining Theorems ~\ref{thm:reduction} and \ref{thm:Liepr-growth}, the complexity of $\Lie(n)$ can be computed for any $n$.
This proves the Erdmann-Lim-Tan Conjecture \cite{ErdmannLimTan}.

\begin{cor}For all $n \in \mathbb{N}, c_{\Sigma_n}(\Lie(n)) =t$ where $p^t \mid n$ and $p^{t+1} \nmid n$.
\end{cor}

\end{document}